\documentclass[reqno,a4paper,draft]{amsart}

\usepackage{enumitem}
\setenumerate{label=\textnormal{(\arabic*)}}

\usepackage{amsmath,amssymb,dsfont,verbatim,bm,array,mathtools}
\usepackage[latin1]{inputenc}
\usepackage{booktabs}
\usepackage[raggedright]{titlesec}
\usepackage{mathtools}

\titleformat{\chapter}[display]
{\normalfont\huge\bfseries}{\chaptertitlename\\thechapter}{20pt}{\Huge}
\titleformat{\section}
{\normalfont\Large\bfseries\center}{\thesection}{1em}{}
\titleformat{\subsection}
{\normalfont\large\bfseries}{\thesubsection}{1em}{}
\titleformat{\subsubsection}[runin]
{\normalfont\normalsize\bfseries}{\thesubsubsection}{1em}{}
\titleformat{\paragraph}[runin]
{\normalfont\normalsize\bfseries}{\theparagraph}{1em}{}
\titleformat{\subparagraph}[runin]
{\normalfont\normalsize\bfseries}{\thesubparagraph}{1em}{}
\titlespacing*{\chapter} {0pt}{50pt}{40pt}
\titlespacing*{\section} {0pt}{3.5ex plus 1ex minus .2ex}{2.3ex plus .2ex}
\titlespacing*{\subsection} {0pt}{3.25ex plus 1ex minus .2ex}{1.5ex plus .2ex}
\titlespacing*{\subsubsection}{0pt}{3.25ex plus 1ex minus .2ex}{1.5ex plus .2ex}
\titlespacing*{\paragraph} {0pt}{3.25ex plus 1ex minus .2ex}{1em}
\titlespacing*{\subparagraph} {\parindent}{3.25ex plus 1ex minus .2ex}{1em}

\input xypic
\xyoption{all}


\newtheorem{theorem}{Theorem}[section]
\newtheorem{lemma}[theorem]{Lemma}
\newtheorem{proposition}[theorem]{Proposition}
\newtheorem{corollary}[theorem]{Corollary}
\newtheorem{conjecture}[theorem]{Conjecture}

\theoremstyle{definition}

\theoremstyle{remark}
\newtheorem{remark}[theorem]{Remark}
\newtheorem{remarks}[theorem]{Remarks}

\DeclareMathOperator{\Jac}{Jac}

\begin{document}
\title{Ideas about the Jacobian Conjecture}

\author{Vered Moskowicz}
\address{Department of Mathematics, Bar-Ilan University, Ramat-Gan 52900, Israel.}
\email{vered.moskowicz@gmail.com}

\begin{abstract}
Let 
$F:\mathbb{C}[x_1,\ldots,x_n] \to \mathbb{C}[x_1,\ldots,x_n]$ 
be a $\mathbb{C}$-algebra endomorphism 
that has an invertible Jacobian.

We bring two ideas concerning the Jacobian Conjecture:
First, we conjecture that for all $n$,
the degree of the field extension
$\mathbb{C}(F(x_1),\ldots,F(x_n)) \subseteq \mathbb{C}(x_1,\ldots,x_n)$
is less than or equal to $d^{n-1}$, 
where $d$ is the minimum of the degrees of the $F(x_i)$'s.
If this conjecture is true,
then the generalized Jacobian Conjecture is true.

Second, we suggest to replace in some known theorems
the assumption on the degrees of the $F(x_i)$'s
by a similar assumption on the degrees of the 
minimal polynomials of the $x_i$'s over $\mathbb{C}(F(x_1),\ldots,F(x_n))$;
this way we obtain some analogous results to the known ones.
\end{abstract}

\maketitle

\section{Introduction}
Throughout this note,
$F:\mathbb{C}[x_1,\ldots,x_n] \to \mathbb{C}[x_1,\ldots,x_n]$ 
will always denote a $\mathbb{C}$-algebra endomorphism 
that satisfies
$\Jac(F(x_1),\ldots,F(x_n)) \in \mathbb{C}^*$.
Denote, $F_i=F(x_i)$, $1 \leq i \leq n$.

The famous $n$-dimensional Jacobian Conjecture,
denote it by $(\mathbb{C},n)$-JC 
or by $n$-JC,
raised by O. H. Keller ~\cite{keller} in 1939,
says that such $F$ is an automorphism.
It is easy to see that the $1$-JC is true.
However, it seems that for every $n \geq 2$, the $n$-JC is still open.
For more details, see, for example, ~\cite{bcw}, ~\cite{essen believe}, 
~\cite{essen affine} and ~\cite{essen book}.
The generalized Jacobian Conjecture,
denote it by $(\mathbb{C},\infty)$-JC
or by $\infty$-JC, is: 
The $n$-JC is true for all $n \geq 1$.

We bring two main ideas concerning the Jacobian Conjecture.

The first idea: We conjecture ``the $n$-Degree Conjecture",
namely, we conjecture that the degree of the field extension
$\mathbb{C}(F_1,\ldots,F_n) \subseteq 
\mathbb{C}(x_1,\ldots,x_n)$ 
is less than or equal to $d^{n-1}$,
where $d$ is the minimum of the degrees of 
$F_1,\ldots,F_n$.
Then we show in Proposition \ref{deg implies JC}
that if the $n$-Degree Conjecture is true for all $n$, 
call this ``the $\infty$-Degree Conjecture",
then the $\infty$-JC is true.
The converse is trivially true,
hence the $\infty$-Degree Conjecture 
is equivalent to the $\infty$-JC.
(Caution: If the $n$-JC is true, then trivially the
$n$-Degree Conjecture is true, but not vice-versa,
as far as we can see).
This idea was inspired by the main result in Y. Zhang's thesis ~\cite{zhang},
which says that, in our words, the $2$-Degree Conjecture is true,
and by van den Essen's observation ~\cite[page 287, lines 6-9]{essen believe}
about a theorem of Lang and Maslamani ~\cite{lang maslamani}.

Without using Essen's observation we have 
Proposition \ref{without essen} and its corollary 
\ref{without essen dru}; 
for the proof of the corollary we use the density theorem ~\cite[Theorem 1.4 (ii)]{ttt}
and take the opportunity to shortly discuss other density theorems
(which are independent of the $\infty$-Degree Conjecture). 

The second idea: We replace in some known theorems
the assumption on the degrees of the $F_i$'s
by a similar assumption on the degrees of the 
minimal polynomials of the $x_i$'s over $\mathbb{C}(F(x_1),\ldots,F(x_n))$.
Denote the degree of the minimal polynomial of $x_i$ over 
$\mathbb{C}(F_1,\ldots,F_n)$
by $d_i$, $1 \leq i \leq n$.
\begin{itemize}
\item Theorem \ref{my thm wang 2}, inspired by Wang's quadratic case theorem 
~\cite[Theorem 62]{wang}, says:
If there exists a subset 
$A \subset \{1,\ldots,n\}$ of order $n-1$
such that for every $j \in A$,
$d_j \leq 2$,
then $F$ is an automorphism.
\item Theorem \ref{my thm magnus n}, inspired by Magnus' theorem 
~\cite{magnus} ~\cite[Theorem 10.2.24]{essen book}, says:
If $\gcd(d_u,d_v)=1$, $1 \leq u \neq v \leq n$,
for each of the $n \choose 2$ pairs of degrees,
then $F$ is an automorphism.
Notice that Magnus' theorem deals with $n=2$,
while Theorem \ref{my thm magnus n} deals with any $n \geq 2$.
\item When $n=2$, we have Theorem \ref{my thm magnus generalized}
inspired by 
~\cite[Corollary 10.2.25 and Theorem 10.2.26]{essen book}:
If $\gcd(d_1,d_2)=1$ or $\gcd(d_1,d_2)=p$, 
for some prime $p \geq 2$,
and the coefficients of the minimal
polynomial of $x_1$ over 
$\mathbb{C}[F_1,F_2]$ are all symmetric or skew-symmetric
with respect to the exchange involution 
$x_1 \mapsto x_2$, $x_2 \mapsto x_1$,
then $F$ is an automorphism.
\end{itemize}
\section{Preliminaries}
We wish to recall the following nice theorems
which we will use throughout this note;
truly, this note would not have been existed without those theorems.

In every theorem we will assume that
$F: \mathbb{C}[x_1,\ldots,x_n] \to \mathbb{C}[x_1,\ldots,x_n]$
is a $\mathbb{C}$-algebra endomorphism 
that satisfies
$\Jac(F_1,\ldots,F_n) \in \mathbb{C}^*$.

Three theorems are only dealing with $n=2$: Zhang's theorem,
the theorem of Cheng-Macky-Wang
and Magnus' theorem (and its corollary and its generalizations).

In some theorems $\mathbb{C}$ can be replaced by a more general ring which is:
\begin{itemize}
\item an algebraically closed field of characteristic zero: Zhang's theorem,
a theorem of Cheng-Macky-Wang (they were working over 
$\mathbb{C}$; we can show that their theorem is still valid over
any algebraically closed field of characteristic zero).
\item a field of characteristic zero: Formanek's theorem, Magnus' theorem
and its corollary, the Galois case.
\item any field: Keller's theorem
(~\cite[Corollary 1.1.35]{essen book} and ~\cite[Theorem 2.1]{bcw}).
\item a UFD with $2 \neq 0$: Wang's quadratic case theorem, Wang's intersection theorem.
\item a commutative $\mathbb{Q}$-algebra:
A theorem of Bass-Connell-Wright-Yagzhev-Dru{\.z}kowski.
\end{itemize}

However, working over $\mathbb{C}$ is good enough
since if the $(\mathbb{C},n)$-JC is true,
then the $(R,n)$-JC is true,
where $R$ is any domain of characteristic zero;
see ~\cite[Proposition 1.1.12 and Lemma 1.1.14]{essen book}
(~\cite[Proposition 1.1.12]{essen book} is a result of Bass, Connell and Wright 
~\cite{bcw}).

\textbf{(1) Keller's theorem} ``the birational case"
~\cite{keller}, ~\cite[Corollary 1.1.35]{essen book} 
and ~\cite[Theorem 2.1]{bcw}:
If
$\mathbb{C}(F_1,\ldots,F_n)=\mathbb{C}(x_1,\ldots,x_n)$,
then $F$ is an automorphism.

In other words,
if the degree of the field extension
$\mathbb{C}(F_1,\ldots,F_n) \subseteq \mathbb{C}(x_1,\ldots,x_n)$
is $1$, then $F$ is an automorphism
(and obviously,
if $F$ is an automorphism,
then the degree of that field extension
is $1$).

\textbf{(2) Zhang's theorem} ~\cite{zhang}:
The degree of the field extension
$\mathbb{C}(F_1,F_2) \subseteq \mathbb{C}(x_1,x_2)$
is less than or equal to the minimum of
the degrees of $F_1$ and $F_2$.
A generalization of Zhang's theorem, still in $\mathbb{C}[x_1,x_2]$ 
but without demanding that the Jacobian will be invertible, 
can be found in ~\cite{kat}, see also ~\cite{formanek2}.

\textbf{(3) Formanek's field of fractions theorem}
(we will also quote another theorem of Formanek)
~\cite[Theorem 2]{formanek observations}:
$\mathbb{C}(F_1,\ldots,F_n,x_1,\ldots,x_{n-1})=
\mathbb{C}(x_1,\ldots,x_n)$.
The $n=2$ case was already proved by Moh ~\cite[page 151]{moh} and by 
Hamann ~\cite[Lemma 2.1, Proposition 2.1(2)]{hamann}
(Moh and Hamann assumed that $F_1$ is monic in $x_2$, 
but this is really not a restriction).

\textbf{(4) Wang's quadratic case theorem}
(we will also quote another theorem of Wang)
~\cite[Theorem 62]{wang}:
If the degree of each $F_i$
is $\leq 2$, then $F$ is an automorphism.
See also ~\cite{oda}, ~\cite[Lemma 3.5]{wright}, ~\cite[Theorem 2.4]{bcw} 
and ~\cite[page 8, Theorem 2.2]{essen affine}.

\textbf{(5) A theorem of Bass-Connell-Wright-Yagzhev-Dru{\.z}kowski}
~\cite{bcw} ~\cite{yag} ~\cite{dru},
see also ~\cite[page 125]{essen book} and 
~\cite[Theorem 4.3]{dru3}:
If for all $n \geq 2$,
$F$ of the following form is an automorphism,
then the $\infty$-JC is true:
$F_i = x_i + H_i$, $1 \leq i \leq n$
where each $H_i$ is either zero or homogeneous of degree $3$
(Dru{\.z}kowski: 
where each $H_i$ is either zero or a third power of a linear form).

\textbf{(6) Essen's observation}
~\cite[page 287, lines 6-9]{essen believe}:
If for all $n \geq 2$,
$F$ of the following form is an automorphism,
then the $\infty$-JC is true:
$F_i = l_i$,
$1 \leq i \leq r$, $l_i$ is linear,
and 
$F_i= x_i + M_i$,
$r+1 \leq i \leq n$, $M_i$ is a monomial.

Essen mentions this, without a proof, as a consequence of a theorem of Lang and Maslamani 
~\cite{lang maslamani}; their theorem says that
$F$ of the following form is an automorphism:
$F_i= x_i + \lambda_i M_i$,
$1 \leq i \leq n$, 
$\lambda_i \in \mathbb{C}$,
and $M_i$ is a monomial.

\textbf{(7) A theorem of Cheng-Macky-Wang; C-M-W's theorem} ~\cite[Theorem 1]{wang younger}:
If $h \in \mathbb{C}[x_1,x_2]$ satisfies $\Jac(F_1,h)= 0$, then $h \in \mathbb{C}[x_1]$.

In other words, C-M-W's theorem says that ``the centralizer with respect to the Jacobian"
of an element $A \in \mathbb{C}[x_1,x_2]$ which has a Jacobian mate
equals $\mathbb{C}[A]$,
where a Jacobian mate of an element $A \in \mathbb{C}[x_1,x_2]$ is
an element $B \in \mathbb{C}[x_1,x_2]$ such that $\Jac(A,B) \in \mathbb{C}^*$.
(Here $A=F_1$ and $B=F_2$).
See also ~\cite[Exercise 2.2.5]{essen book}.

\textbf{(8) Density theorem;} mentioned by Truong 
~\cite[Theorem 1.4 (ii)]{ttt}, see also ~\cite{cima}:
If there exists a dense subset of Dru{\.z}kowski mappings,
in which every element is an automorphism, then the 
$\infty$-JC is true. 

\textbf{(9) Wang's intersection theorem} ~\cite[Theorem 41 (i)]{wang}
~\cite[Corollary 1.1.34 (ii)]{essen book}:
 $\mathbb{C}(F_1,\ldots,F_n) \cap \mathbb{C}[x_1,\ldots,x_n] = 
\mathbb{C}[F_1,\ldots,F_n]$.
A generalization of it, due to Bass, can be found in
~\cite[Remark after Corollary 1.3, page 74]{bass} ~\cite[Proposition D.1.7]{essen book}.

\textbf{(10) A theorem of Jedrzejewicz and Zieli\'{n}ski} 
~\cite[Theorem 3.6]{J}:
Let $A$ be a UFD.
Let $R$ be a subring of A such that
$R^*=A^*$ and 
$Q(R) \cap A = R$.
If $R$ is square-factorially closed in $A$,
then $R$ is root closed in $A$.

A subring $R$ of a UFD $A$ is called \textit{square-factorially closed} in $A$,
if for every $a \in A$ and square-free
$b \in A$ (see ~\cite[page 5]{J}), 
$a^2 b \in R-0$ implies that $a,b \in R-0$;
see ~\cite[Definition 3.5]{J}.

A subring $R$ of a ring $A$ is called \textit{root closed} in $A$,
if for every $a \in A$ and every $m \geq 1$,
if $a^m \in R$, then $a \in R$.

Based on the results in ~\cite{J},
it is not difficult to obtain the following lemma:
\begin{lemma}\label{lemma}
Let
$F:\mathbb{C}[x_1,\ldots,x_n] \to \mathbb{C}[x_1,\ldots,x_n]$
be a $\mathbb{C}$-algebra endomorphism  
that satisfies
$\Jac(F_1,\ldots,F_n) \in \mathbb{C}^*$.
Then $\mathbb{C}[F_1,\ldots,F_n]$ 
is root closed in 
$\mathbb{C}[x_1,\ldots,x_n]$.
\end{lemma}

Hamann in ~\cite[Proposition 2.11]{hamann} 
has proved a special case of Lemma \ref{lemma}.

\begin{proof}
Follows immediately from 
~\cite[Theorem 2.4, Theorem 3.4, Theorem 3.6]{J} which we can apply here
thanks to Wang's intersection theorem.
\end{proof}

\textbf{(11) Formanek's theorem} ~\cite[Theorem 1]{formanek two notes}:
Suppose that there is a polynomial 
$F_{n+1}$ in $\mathbb{C}[x_1,\ldots,x_n]$ 
such that
$\mathbb{C}[F_1,\ldots,F_n,F_{n+1}]=
\mathbb{C}[x_1,\ldots,x_n]$.
Then $F$ is an automorphism.

\textbf{(12) Magnus' theorem} ~\cite{magnus}:
If the greatest common divisor of 
the degrees of $F_1$ and $F_2$
is $1$, then $F$ is an automorphism,
and the degree of $F_1$ is $1$
or the degree of $F_2$ is $1$.

Its corollary ~\cite[Corollary 10.2.25]{essen book} is:
If the degree of $F_1$ or 
the degree of $F_2$
is a prime number, 
then $F$ is an automorphism.

More results in that direction are: 
$F$ is an automorphism, if the greatest common divisor of the degrees of 
$F_1$ and $F_2$ is:
\begin{itemize}
\item $\leq 2$; Nakai and Baba ~\cite{baba nakai}.
\item $\leq 8$ or a prime number;
Appelgate and Onishi ~\cite{app} and Nagata ~\cite{nagata}.
\end{itemize}
Those results can be found in Essen's book ~\cite[pages 254-256]{essen book}
(see also Moh's theorem ~\cite[Theorem 10.2.30]{essen book} ~\cite{moh}).

\textbf{(13) The Galois case:}
If 
$\mathbb{C}(x_1,\ldots,x_n)$ is Galois over 
$\mathbb{C}(F_1,\ldots,F_n)$,
then $F$ is an automorphism.

This was first proved by Campbell ~\cite{camp}.
The $n=2$ case was proved by Abhyankar ~\cite[Theorem 21.11]{abhy}.
Other proofs are due to Razar ~\cite{razar},
Wright ~\cite{wright} and Oda ~\cite{oda}.

This result is equivalent to other results,
for example, it is equivalent to Keller's theorem 
and to the integral case;
see ~\cite[Theorem 46]{wang}, 
~\cite[Theorem 8]{wang deri}, 
~\cite[Theorem 2.1]{bcw} 
and ~\cite[Theorem 2.2.16]{essen book}.


\section{First idea: The $\infty$-Degree Conjecture implies the $\infty$-JC}

Inspired by Zhang's theorem, we conjecture the following:
\begin{conjecture}[The $n$-Degree Conjecture]
Let 
$F:\mathbb{C}[x_1,\ldots,x_n] \to \mathbb{C}[x_1,\ldots,x_n]$ 
be a $\mathbb{C}$-algebra endomorphism 
that has an invertible Jacobian.
Let $d$ be the minimum of the degrees of 
$F_1,\ldots,F_n$.
Then the degree of the field extension
$\mathbb{C}(F_1,\ldots,F_n) \subseteq 
\mathbb{C}(x_1,\ldots,x_n)$ 
is less than or equal to $d^{n-1}$.
\end{conjecture}

\begin{remark}
Wang's Degree Conjecture ~\cite[Degree Conjecture 63]{wang}
is different from our $n$-Degree Conjecture. 
Wang's Degree Conjecture is connected to Wang's quadratic case theorem, 
and was shown in ~\cite[Corollary 1.4]{bcw} to be true.
\end{remark}

The $2$-Degree Conjecture has an affirmative answer,
due to Zhang's theorem.
We do not know yet if the $n \geq 3$ case also
has an affirmative answer.
On the one hand, Zhang and Katsylo have not mentioned higher dimensions,
maybe because it seemed to them difficult to generalize their results
to higher dimensions.
On the other hand, if Zhang had known about Formanek's
field of fractions theorem and Abhyankar's result ~\cite[Corollary 2.6]{abhy2}
(which is a corollary to Bertini Lemma ~\cite[Bertini Lemma 2.5]{abhy2};
Zhang used results of Bertini in his proof), 
both published 3 years later,
then perhaps he would have generalized his result to $n \geq 3$.

We now bring a sketch of proof to the $n$-Degree Conjecture.
We do not claim that it is a proof! It is just what we were able to 
obtain from Zhang's theorem and other theorems, and it contains some gaps.
If the $\infty$-JC is true, then we think that it is possible to
prove those gaps, but if the $\infty$-JC is false,
then by Proposition \ref{deg implies JC} those gaps do not have a proof.

\begin{proof}[A sketch of Proof]

\textbf{First observation:}
In order to prove the $n$-Degree Conjecture for $n \geq 3$
(if it is true), it suffices to prove that the degree of the minimal polynomial of $x_i$ over 
$\mathbb{C}(F_1,\ldots,F_n)$ is less than or equal to $d$,
for every $1 \leq i \leq n-1$.
Indeed, assume that we have proved that the degree of the minimal polynomial of $x_i$ over 
$\mathbb{C}(F_1,\ldots,F_n)$ is less than or equal to $d$.
{}From Formanek's field of fractions theorem we have
$\mathbb{C}(F_1,\ldots,F_n,x_1,\ldots,x_{n-1})=
\mathbb{C}(x_1,\ldots,x_n)$.
Therefore, by the multiplicativity of degrees of field extensions,
we obtain that the degree of the field extension
$\mathbb{C}(F_1,\ldots,F_n)\subseteq
\mathbb{C}(x_1,\ldots,x_n)$
is less than or equal to $d^{n-1}$.

We hope that it is possible to prove that
the degree of the minimal polynomial of $x_i$ 
over $\mathbb{C}(F_1,\ldots,F_n)$ is less than or equal to $d$.

\textbf{Second observation:}
Let us explain why we think that 
Formanek's field of fractions theorem and Abhyankar's result
may help in generalizing Zhang's theorem to $n \geq 3$.
~\cite[Lemma 2]{zhang} says the following:
Let $k$ be an algebraically closed field of characteristic zero,
and let $p(x_1,x_2), q(x_1,x_2) \in k[x_1,x_2]$ such that:
\begin{itemize}
\item [(i)] $p$ and $q$ are algebraically independent over $k$ 
and monic in $x_2$.
\item [(ii)] The field $k(q)$ is algebraically closed in the field $k(x_1,x_2)$.
\item [(iii)] $k(p,q,x_1)=k(x_1,x_2)$.
\end{itemize}
Then there exists a constant $c \in k$ such that
$[k(x_1,x_2):k(p,q)] = 
[k(\bar{x_1},\bar{x_2}):k(p(\bar{x_1},\bar{x_2}))]$,
where $\bar{x_1}$ and $\bar{x_2}$ are the images of $x_1$ and $x_2$
under the quotient map:
$k[x_1,x_2] \to k[x_1,x_2]/(q(x_1,x_2)-c)$.

Now, let
$F: \mathbb{C}[x_1,\ldots,x_n] \to \mathbb{C}[x_1,\ldots,x_n]$
be a $\mathbb{C}$-algebra endomorphism 
that satisfies
$\Jac(F_1,\ldots,F_n) \in \mathbb{C}^*$.

\begin{itemize}
\item [(i)] $F_1,\ldots,F_n$ are algebraically independent over $\mathbb{C}$,
since the Jacobian is non-zero (see, for example, ~\cite[pages 19-20]{makar} or
~\cite[Proposition 6A.4]{rowen}).
It is easy to arrange that $F_1,\ldots,F_n$ are monic in $x_n$; 
just multiply $F$ by a suitable automorphism $G$: 
$G(x_i)=x_i+x_n^{m_i}$, 
$1 \leq i \leq n-1$,
$G(x_n)=x_n$,
and then work with $FG$ instead of $F$.

\item [(ii)] In Abhyankar's notations 
~\cite[Bertini Lemma 2.5, Corollary 2.6]{abhy2}
take:
$F=\mathbb{C}$, $K=\mathbb{C}(F_1,\ldots,F_n)$,
$E=\mathbb{C}(x_1,\ldots,x_n)$.
Obviously, $\mathbb{C}$ is algebraically closed in 
$\mathbb{C}(x_1,\ldots,x_n)$.
{}From ~\cite[Corollary 2.6]{abhy2} there exist
$y_1,\ldots,y_n$,
linear combinations of $F_1,\ldots,F_n$,
such that
$\mathbb{C}[y_1,\ldots,y_n]=
\mathbb{C}[F_1,\ldots,F_n]$
and 
$\mathbb{C}(y_2,\ldots,y_n)$
is algebraically closed in 
$\mathbb{C}(x_1,\ldots,x_n)$.
$y_1,\ldots,y_n$
are algebraically independent,
because 
$\mathbb{C}[y_1,\ldots,y_n]=
\mathbb{C}[F_1,\ldots,F_n]$.
Denote
$y_i= c_{i,1}F_1+\ldots+c_{i,n}F_n$,
$c_{i,j} \in \mathbb{C}$
$1 \leq i,j \leq n$.
Define the following automorphism $H$:
$H(x_i)= c_{i,1}x_1+\ldots+c_{i,n}x_n$,
$1 \leq i,j \leq n$.
Then work with $FH$ instead of $F$:
$(FH)(x_i)=F(H(x_i))=
F(c_{i,1}x_1+\ldots+c_{i,n}x_n)=
c_{i,1}F(x_1)+\ldots+c_{i,n}F(x_n)=
c_{i,1}F_1+\ldots+c_{i,n}F_n= y_i$.
By the Chain Rule, the Jacobian of $FH$ is also in $\mathbb{C}^*$.
$\mathbb{C}(y_2,\ldots,y_n)$
is algebraically closed in 
$\mathbb{C}(x_1,\ldots,x_n)$,
and this is exactly what we needed
to prove.
So if we return to our previous notations,
we obtained that
$\mathbb{C}(F_2,\ldots,F_n)$
is algebraically closed in 
$\mathbb{C}(x_1,\ldots,x_n)$.

Observe that Zhang used other methods in his proof for the $n=2$ case,
see ~\cite[Lemma 4]{zhang}.

\item [(iii)] $\mathbb{C}(F_1,\ldots,F_n,x_1,\ldots,x_{n-1})=
\mathbb{C}(x_1,\ldots,x_n)$
by Formanek's field of fractions theorem.
\end{itemize}

\textbf{Third observation:}
What should be the conclusion in a generalized ~\cite[Lemma 2]{zhang}?
In the second observation we have seen that 
the three generalized conditions $(i),(ii),(iii)$ are satisfied.
Perhaps the conclusion is:
There exists a constant $c \in \mathbb{C}$ 
such that
$[\mathbb{C}(x_1,\ldots,x_n):\mathbb{C}(F_1,\ldots,F_n)] = 
[\mathbb{C}(\bar{x_1},\ldots,\bar{x_n}):\mathbb{C}(F_1(\bar{x_1},\ldots,\bar{x_n}))]$,
where $\bar{x_1},\ldots,\bar{x_n}$ 
are the images of $x_1,\ldots,x_n$
under the map:
$\mathbb{C}[x_1,\ldots,x_n] 
\to \mathbb{C}[x_1,\ldots,x_n]/(F_2(x_1,\ldots,x_n) \cdots F_n(x_1,\ldots,x_n) - c)$.

One has to be careful if this is really the correct generalization.
(It seems that it is necessary that the quotient ring will be an integral domain,
hence the ideal 
$(F_2(x_1,\ldots,x_n) \cdots F_n(x_1,\ldots,x_n) - c)$
needs to be a prime ideal,
so its generator needs to be irreducible).

\textbf{Fourth observation:}
The arguments in ~\cite[pages 14-15]{zhang} should be generalized.
We hope that the results in algebraic function theory can be applied here also
(this seems quite difficult; especially generalizing ~\cite[III, page 10]{zhang}).

\end{proof}

\begin{remarks}[Two remarks about the second observation in the sketch of proof]

\textbf{I} The generalized condition $(i)$ and 
the generalized condition $(ii)$
can be satisfied simultaneously,
despite the fact that in $(i)$
we take $FG$ and in $(ii)$ we take $FH$
($G$ is a triangular automorphism, while 
$H$ is a linear automorphism; 
most probably $G \neq H$).
More precisely: 
First, we work with
$FG$ instead of $F$ ($G$ is as in $(i)$),
and then we move to work with
$FGH$ instead of $FG$ ($H$ is as in $(ii)$).
By our choice of $G$,
$(FG)(x_1),\ldots,(FG)(x_n)$ are monic in $x_n$.
Write
$(FG)(x_i)= e_i x_n^{m_i} + \cdots$,
where $e_i \in \mathbb{C}$,
$m_i \geq 0$
and
$\cdots = \sum_{j=0}^{m_i - 1} e_{i,j}x_n^j$
with
$e_{i,j} \in \mathbb{C}[x_1,\ldots,x_{n-1}]$.
We claim that 
$(FGH)(x_1),\ldots,(FGH)(x_n)$ are still monic in $x_n$;
more accurately, one can find a suitable linear automorphism
$H$ (which depends on $G$) such that 
$(FGH)(x_1),\ldots,(FGH)(x_n)$ are still monic in $x_n$.
Indeed, we have
$(FGH)(x_i)=(FG)(H(x_i))=
(FG)(c_{i,1}x_1+\ldots+c_{i,n}x_n)=
c_{i,1}(FG)(x_1)+\ldots+c_{i,n}(FG)(x_n)$. 
W.l.o.g, 
$m_1 \geq m_2 \geq \ldots \geq m_n$
and 
$m:= m_1=\ldots=m_r$, for some $r \geq 1$.
Hence we have,
$(FGH)(x_i)=
c_{i,1}(e_1 x_n^{m_1} + \cdots)+\ldots+
c_{i,n}(e_n x_n^{m_n} + \cdots)=
c_{i,1}e_1 x_n^{m} + \ldots + c_{i,r}e_r x_n^{m} + \cdots=
c_i x_n^{m} + \cdots$,
where 
$c_i := c_{i,1}e_1 + \ldots + c_{i,r}e_r$,
and
$\cdots$ has $x_n$-degree strictly less than $m$.
There are two options:
\begin{itemize}
\item For every $1 \leq i \leq n$,
$c_i \neq 0$.
Then for every $1 \leq i \leq n$,
$(FGH)(x_i)= c_i x_n^{m} + \cdots$
is monic in $x_n$.
\item There exists a subset 
$A \subseteq \{1,\ldots,n \}$
such that for every $l \in A$,
$c_l = 0$.
Hence, for an ``arbitrary" $H$
it may happen that some of the
$(FGH)(x_j)$'s are not monic in $x_n$.
Luckily, since there are infinitely many
good options for a linear automorphism $H$
as in $(ii)$, 
it is possible to take a suitable $H$
for which all the $c_i$'s are non-zero.
\end{itemize}

\textbf{II} In the generalized condition $(ii)$
we have mentioned Bertini-Abhyankar's result 
~\cite[Bertini Lemma 2.5, Corollary 2.6]{abhy2}
which is relevant to ~\cite[First section]{J}.
However, when $n=3$ for example,
we are only able to obtain that
for some $\lambda \in \mathbb{C}$
(actually, for infinitely many 
$\mathbb{C} \ni \lambda$'s),
$\mathbb{C}[x_3][F_1+\lambda F_2]$,
is algebraically closed in 
$\mathbb{C}[x_1,x_2,x_3]$,
but not that 
$\mathbb{C}[F_1,F_2]$
is algebraically closed in 
$\mathbb{C}[x_1,x_2,x_3]$,
where 
$F_1,F_2 \in \mathbb{C}[x_1,x_2]$
with $\Jac(F_1,F_2) \in \mathbb{C}^*$.
Indeed, take
$F= \mathbb{C}(x_3)$,
$K=\mathbb{C}(x_3)(F_1,F_2)$
and
$E=\mathbb{C}(x_3)(x_1,x_2)$.
The proof of ~\cite[Bertini Lemma 2.5]{abhy2}
shows that
$\mathbb{C}(x_3)(F_1+\lambda F_2)$
is algebraically closed in 
$\mathbb{C}(x_1,x_2,x_3)$.
Then by 
~\cite[Exercise 1.4 (1)]{daigle}
combined with a result of Bass 
~\cite[Remark after Corollary 1.3, page 74]{bass} ~\cite[Proposition D.1.7]{essen book}
that generalizes Wang's intersection theorem,
we obtain that
$\mathbb{C}[x_3][F_1+\lambda F_2]$
is algebraically closed in 
$\mathbb{C}[x_1,x_2,x_3]$.
\end{remarks}

Next, similarly to the definition of the $\infty$-JC,
we define the $\infty$-Degree Conjecture:

\begin{conjecture}[The $\infty$-Degree Conjecture]
The $n$-Degree Conjecture is true for all $n \geq 2$.
\end{conjecture}

The importance of finding an affirmative answer 
(if possible) to the $\infty$-Degree Conjecture 
is clear from the following proposition:

\begin{proposition}\label{deg implies JC}
If the $\infty$-Degree Conjecture is true,
then the $\infty$-JC is true.
\end{proposition}

\begin{proof}
By Essen's observation it suffices to show that
for all $n \geq 2$,
a $\mathbb{C}$-algebra endomorphism
$F: \mathbb{C}[x_1,\ldots,x_n] \to \mathbb{C}[x_1,\ldots,x_n]$ 
having an invertible Jacobian and of the following form is an automorphism:
$F_i = l_i$,
$1 \leq i \leq r$, $l_i$ is linear,
and 
$F_i= x_i + M_i$,
$r+1 \leq i \leq n$, $M_i$ is a monomial.

Fix $n \geq 2$.
We show that such $F$ is an automorphism;
indeed for such $F$ we have $d=1$, 
because $F_1$ has degree $1$.
We assumed that the $\infty$-Degree Conjecture is true, 
hence the degree of the field extension
$\mathbb{C}(F_1,\ldots,F_n) \subseteq 
\mathbb{C}(x_1,\ldots,x_n)$ 
is $1$,
so, 
$\mathbb{C}(F_1,\ldots,F_n) =
\mathbb{C}(x_1,\ldots,x_n)$.
Finally, Keller's theorem
implies that $F$ is an automorphism.
\end{proof}

Trivially, the converse of Proposition \ref{deg implies JC}
is also true;
actually it is slightly stronger than the converse of Proposition \ref{deg implies JC}.
Notice that in Proposition \ref{deg implies JC}
it was critical to consider all $n \geq 2$ simultaneously,
since we applied Essen's observation;
even for $n=2$ we do not see why the $2$-Degree Conjecture
(which is true) should imply the two-dimensional Jacobian Conjecture.

\begin{proposition}\label{JC implies degree}
If the $n$-JC is true,
then the $n$-Degree Conjecture is true.
Therefore, if the $\infty$-JC is true,
then the $\infty$-Degree Conjecture is true.
\end{proposition}

\begin{proof}
Let  
$F: \mathbb{C}[x_1,\ldots,x_n] \to \mathbb{C}[x_1,\ldots,x_n]$ 
be a $\mathbb{C}$-algebra endomorphism
having an invertible Jacobian.
The $n$-JC is true, so $F$ is an automorphism:
$\mathbb{C}[F_1,\ldots,F_n] = \mathbb{C}[x_1,\ldots,x_n]$.
Then 
$\mathbb{C}(F_1,\ldots,F_n) = \mathbb{C}(x_1,\ldots,x_n)$,
so the degree of the field extension
$\mathbb{C}(F_1,\ldots,F_n) \subseteq 
\mathbb{C}(x_1,\ldots,x_n)$ 
equals $1$, 
and trivially $1 \leq d^{n-1}$
($d$ is the minimum of the degrees of 
$F_1,\ldots,F_n$,
so $d \geq 1$).
The second statement is clear.
\end{proof}

Therefore we have:
\begin{theorem}\label{deg equals JC}
The $\infty$-Degree Conjecture is equivalent to the $\infty$-JC.
\end{theorem}

\subsection{First idea without Essen's observation; density theorem}
Without using Essen's observation,
we can have the following proposition:

\begin{proposition}\label{without essen}
Assume that the $n$-Degree Conjecture is true.
If a $\mathbb{C}$-algebra endomorphism 
$F:\mathbb{C}[x_1,\ldots,x_n] \to \mathbb{C}[x_1,\ldots,x_n]$ 
has an invertible Jacobian
and if (at least) one of its $F_j$'s has degree $1$,
then $F$ is an automorphism.
\end{proposition}

\begin{proof}
The assumption that (at least) one of the $F_j$'s has degree $1$
implies that $d=1$.
By assumption the $n$-Degree Conjecture is true,
hence the degree of the field extension
$\mathbb{C}(F_1,\ldots,F_n) \subseteq 
\mathbb{C}(x_1,\ldots,x_n)$ 
is $\leq d^{n-1}=1^{n-1}=1$.
So, 
$\mathbb{C}(F_1,\ldots,F_n) =
\mathbb{C}(x_1,\ldots,x_n)$,
and by Keller's theorem $F$ is an automorphism.
\end{proof}

Let 
$F:\mathbb{C}[x_1,x_2] \to \mathbb{C}[x_1,x_2]$ 
be a $\mathbb{C}$-algebra endomorphism 
that has (w.l.o.g.) Jacobian $1$,
and let 
$F_1=\alpha x_1+ \beta x_2+ u$,
where $\alpha, \beta ,u \in \mathbb{C}$.
It is not difficult to show that  
$F$ is an automorphism,
without using Zhang's theorem
(= the $2$-Degree Conjecture is true).
Indeed, define
$g(x_1):=\alpha x_1+ \beta x_2+ u = F_1$
and 
$g(x_2):=\gamma x_1+ \delta x_2+ v$,
where $\gamma, \delta, v \in \mathbb{C}$,
such that $\alpha \delta- \beta \gamma =1$. 
Trivially, $g$ is an automorphism.
Now,
$0=1-1=\Jac(F_1,F_2)-\Jac(g(x_1),g(x_2))=
\Jac(F_1, F_2-g(x_2))$.
By C-M-W's theorem,
$F_2-g(x_2) \in \mathbb{C}[F_1]$.
Hence,
$F_2 = g(x_2) + \sum c_i F_1^i$
for some $c_i \in \mathbb{C}$.
It is clear that
$F_1= \alpha x_1+ \beta x_2+ u$,
$F_2 = \gamma x_1+ \delta x_2+ v + \sum c_i F_1^i$
with $\alpha \delta- \beta \gamma =1$
is an automorphism.

However, when $n \geq 3$ we are not familiar with a result
(except our Proposition \ref{without essen}
which assumes the $n$-Degree Conjecture)
which states that if $F$ has an invertible Jacobian
and if (at least) one of its $F_j$'s has degree $1$,
then $F$ is an automorphism.
Even if some kind of a generalized version to C-M-W's theorem when $n \geq 3$
do exist 
(this is probably connected to the kernel conjecture,
see, for example, ~\cite[Conjecture 6.1]{essen affine}),
it seems that it will not be enough to demand that one of the $F_j$'s 
has degree $1$ and instead one should demand that 
$n-1$ $F_j$'s will have degree $1$; but this is too restrictive.

As a corollary to Proposition \ref{without essen}
and thanks to the density theorem,
we have:
\begin{corollary}\label{without essen dru}
Assume that the $\infty$-Degree Conjecture is true.
Let $D_1$ be the subset of Dru{\.z}kowski mappings 
such that every element in $D_1$ has (at least) 
one of its $F_j$'s of degree $1$.
Assume that $D_1$ is dense in the set of all Dru{\.z}kowski mappings.
Then the $\infty$-JC is true.
\end{corollary}

\begin{proof}
By Proposition \ref{without essen},
every element in $D_1$ is an automorphism.
By assumption $D_1$ is dense, so we can apply 
the density theorem
and get that the $\infty$-JC is true.
\end{proof}

T. T. Truong ~\cite[Remarks after Theorem 1.4]{ttt}
suggests to apply the density theorem ~\cite[Theorem 1.4 (ii)]{ttt}
to a result of D. Yan ~\cite{dan yan} in which every element on the diagonal of $A$ 
is non-zero, while in our Corollary \ref{without essen dru}
(at least) one row of $A$ consists of zeros only,
so (at least) one element on the diagonal of $A$ is zero
(see also ~\cite{dru2}, ~\cite[Theorem 4.4]{dru3} and ~\cite[Theorem 1.1]{dru4}).
We wish to quote ~\cite[page 25, Remarks 2]{ttt}:
``Rusek ~\cite{rusek} showed that the matrices of corank exactly $1$
is not dense in the set of Dru{\.z}kowski matrices".
However, the set of matrices associated to our $D_1$ 
contains matrices of corank $\in \{1,2,3,\ldots,n-1,n \}$,
so there is some hope that $D_1$ is dense in the set of Dru{\.z}kowski matrices.
Anyway (whether $D_1$ is dense or not), if we use Essen's observation, then we get 
Proposition \ref{deg implies JC}, which does not contain any assumptions on density.
The following idea is independent of the 
$\infty$-Degree Conjecture; however, we decided to bring it now
since it is very similar to the density theorem:
Denote by $E_n$ the set of endomorphisms 
$F : \mathbb{C}[x_1,\ldots,x_n] \to \mathbb{C}[x_1,\ldots,x_n]$
having an invertible Jacobian and of the following form:
$F_i = l_i$,
$1 \leq i \leq r$, $l_i$ is linear,
and 
$F_i= x_i + M_i$,
$r+1 \leq i \leq n$, $M_i$ is a monomial.
Denote $E= \cup E_n$.
Essen's observation says that if every element in $E$
is an automorphism,
then the $\infty$-JC is true.
Denote by $L_n$ the subset of $E_n$
consisting of those endomorphisms 
having each $l_i= x_i$,
$1 \leq i \leq r$.
Denote $L= \cup L_n$.
{}From the theorem of Lang and Maslamani,
every element in $L$ is an automorphism.
We imitate the density theorem
and get a ``second density theorem":
If there exists a dense subset of $E$,
in which every element is an automorphism, then the 
$\infty$-JC is true. 
We have not yet tried to prove the second density theorem;
if it is provable, then it may be worth to try to show that
$L$ is dense in $E$.

Another idea independent of the $\infty$-Degree Conjecture is 
~\cite[page 17, 5 lines after Proposition 3.4]{me};
it is an unproved density theorem, which (if true) implies that the
Dixmier Conjecture is true. 
It seems that a similar result can be obtained for $\mathbb{C}[x_1,x_2]$,
namely, a similar density theorem to ~\cite[page 17, 5 lines after Proposition 3.4]{me}, 
which (if true) implies that the two-dimensional Jacobian Conjecture is true.


\section{Second idea: Considering degrees of the minimal polynomials of the $x_i$'s
instead of degrees of the $F_i$'s}

Given a 
$\mathbb{C}$-algebra endomorphism
$F:\mathbb{C}[x_1,\ldots,x_n] \to \mathbb{C}[x_1,\ldots,x_n]$ 
that has an invertible Jacobian,
we denote the degree of the minimal polynomial of $x_i$ over 
$\mathbb{C}(F_1,\ldots,F_n)$
by $d_i$, $1 \leq i \leq n$.

\subsection{Wang's quadratic case theorem}
Inspired by Wang's quadratic case theorem, we bring the following theorem,
which we were able to prove thanks to the theorem of Jedrzejewicz and Zieli\'{n}ski
and Wang's intersection theorem.
\begin{theorem}[Our Wang's quadratic case theorem]\label{my thm wang 2}
Let 
$F:\mathbb{C}[x_1,\ldots,x_n] \to \mathbb{C}[x_1,\ldots,x_n]$ 
be a $\mathbb{C}$-algebra endomorphism 
that has an invertible Jacobian.
If there exists a subset
$A \subset \{1,\ldots,n\}$ of order $n-1$
such that for every $j \in A$,
the minimal polynomial of $x_j$ over 
$\mathbb{C}(F_1,\ldots,F_n)$
has degree $d_j \leq 2$,
then $F$ is an automorphism.
\end{theorem}

If for some $j$,
the minimal polynomial of $x_j$ over 
$\mathbb{C}(F_1,\ldots,F_n)$ 
has degree $=2$,
then $F$ is \textit{not} an automorphism;
indeed, in such case, the degree of the field extension
$\mathbb{C}(F_1,\ldots,F_n)\subseteq
\mathbb{C}(x_1,\ldots,x_n)$
is at least $2$, while for an automorphism 
the degree of that field extension must equal $1$.
Therefore, in Theorem \ref{my thm wang 2} we just assume that
each of those minimal polynomials has degree $\leq 2$.
Perhaps, instead of writing:
``the minimal polynomial of $x_j$ over 
$\mathbb{C}(F_1,\ldots,F_n)$
has degree $d_j \leq 2$",
we should have written:
``there exists a polynomial 
$g_j=g_j(T) \in \mathbb{C}(F_1,\ldots,F_n)[T]$
of degree $\leq 2$ such that
$g_j(x_j)=0$".

Of course, in Wang's quadratic case theorem, it is possible to have
degrees of some $F_j$'s exactly $2$,
for example,
$x \mapsto x$, $y \mapsto y+x^2$, $z \mapsto z+y^2$.

\begin{proof}
W.l.o.g. $A=\{1,\ldots,n-1 \}$.
Fix $j \in A$.
The minimal polynomial of $x_j$ over 
$\mathbb{C}(F_1,\ldots,F_n)$,
denote it $f_j=f_j(T)$,
has degree $d_j \leq 2$.

If $f_j$ is of degree $1$,
then it is of the following form:
$f_j=T+a_j$,
where $a_j \in \mathbb{C}(F_1,\ldots,F_n)$.
Hence $x_j = -a_j$.
Then, by Wang's intersection theorem,
$x_j \in \mathbb{C}[F_1,\ldots,F_n]$.

If $f_j$ is of degree less than or equal to $2$,
then it is of the following form:
$f_j=\tilde{a_j}T^2+\tilde{b_j}T+\tilde{c_j}$,
where $\tilde{a_j},\tilde{b_j},\tilde{c_j} 
\in \mathbb{C}(F_1,\ldots,F_n)$.
Clearly we can get
$a_jx_j^2+b_jx_j+c_j=0$,
for some $a_j,b_j,c_j \in \mathbb{C}[F_1,\ldots,F_n]$.
For convenience, denote
$a=a_j,b=b_j,c=c_j$.
Then
$a^2x_j^2+abx_j+ac=0$,
so
$(ax_j)^2+2(ax_j)(b/2)+b^2/4-b^2/4+ac=0$,
and we have
$(ax_j+b/2)^2=b^2/4-ac \in \mathbb{C}[F_1,\ldots,F_n]$.
By Lemma \ref{lemma},
$\mathbb{C}[F_1,\ldots,F_n]$
is root closed in 
$\mathbb{C}[x_1,\ldots,x_n]$,
therefore,
$ax_j+b/2 \in \mathbb{C}[F_1,\ldots,F_n]$.
So
$ax_j \in \mathbb{C}[F_1,\ldots,F_n]$,
and then
$x_j \in \mathbb{C}(F_1,\ldots,F_n)$,
and again by Wang's intersection theorem,
$x_j \in \mathbb{C}[F_1,\ldots,F_n]$.

We obtained that for every $j \in A$,
$x_j \in \mathbb{C}[F_1,\ldots,F_n]$.

Hence,
$\mathbb{C}[F_1,\ldots,F_n][x_n]=
\mathbb{C}[x_1,\ldots,x_n]$.
Finally, Formanek's theorem
implies that $F$ is an automorphism.
\end{proof}

\begin{remark}\label{remark}
When $n=2$ it is 'immediate' that $F$ is an automorphism:
Indeed, by Formanek's field of fractions theorem,
$\mathbb{C}(F_1,F_2,x_1)=
\mathbb{C}(x_1,x_2)$.
If the minimal polynomial of $x_1$ over 
$\mathbb{C}(F_1,F_2)$ has degree $1$,
then
$\mathbb{C}(F_1,F_2)=
\mathbb{C}(x_1,x_2)$,
and we are done by Keller's theorem.

If the minimal polynomial of $x_1$ over 
$\mathbb{C}(F_1,F_2)$ has degree $2$,
then the field extension
$\mathbb{C}(F_1,F_2) \subseteq
\mathbb{C}(F_1,F_2,x_1)= \mathbb{C}(x_1,x_2)$
is Galois,
since every separable field extension $k(a)/k$
of degree $2$ is Galois:
$(T-a)(T-b)=T^2-(a+b)T+ab \in k[T]$,
so $a+b \in k$, and then
$b \in k(a)$.
Hence $k(a)$ is the splitting field of
$T^2-(a+b)T+ab$ over $k$.
We have recalled in the Preliminaries section that if
$\mathbb{C}(F_1,F_2) \subseteq \mathbb{C}(x_1,x_2)$
is Galois, 
then $F$ is an automorphism.
\end{remark}

Theorem \ref{my thm wang 2} can be thought of as some kind of an analogue result to that of Wang's
quadratic case theorem.
We wish to quote van den Essen ~\cite[page 8, after Theorem 2.2]{essen affine}:
``Now one could think that the above result is just a very special case of the Jacobian Conjecture,
however we have the following theorem, proved independently by Yagzhev
~\cite{yag} and Bass, Connell and Wright ~\cite{bcw}:
If the Jacobian Conjecture holds for all polynomial maps $F$
with the degree of each $F(x_i)$ less or equal $3$,
and all $n \geq 2$, then the generalized Jacobian Conjecture holds!".

In view of this, it may be interesting in Theorem \ref{my thm wang 2}
to replace $\leq 2$ by $\leq 3$ and try to prove that $F$ is an automorphism.
We have not yet succeeded to prove the degree $\leq 3$ case,
only some special cases of it. For example,
$F$ is an automorphism, 
if there exist $n-1$ polynomials
$g_j \in \mathbb{C}[F_1,\ldots,F_n][T]$,
$g_j(x_j)=0$,
$g_j=g_j(T)=a_jT^3+b_jT^2+c_jT+d_j$,
such that each $g_j$ is of one of the following forms:
\begin{itemize}
\item [(1)] $a_j=0$, so we reduce to the degree $\leq 2$ case.
\item [(2)] $b_j^2=3a_jc_j$, 
since in that case, from
$a_jx_j^3+b_jx_j^2+c_jx_j+d_j=0$
we get
$(a_jx_j+b_j/3)^3= b_j^3/27-a_j^2d_j 
\in \mathbb{C}[F_1,\ldots,F_n]$.
Hence, by similar considerations to the ones we have already seen,
we obtain
$x_j \in \mathbb{C}[F_1,\ldots,F_n]$.
\end{itemize}
Generally, from
$a_jx_j^3+b_jx_j^2+c_jx_j+d_j=0$
with $a_j \neq 0$,
we can easily get,
$(a_jx_j+b_j/3)^3+\epsilon (a_jx_j+b_j/3) +\delta =0$,
where $\epsilon, \delta \in 
\mathbb{C}[F_1,\ldots,F_n]$.
More precisely,
$\epsilon= a_jc_j-b_j^2/3$
and 
$\delta= -b_j/3(a_jc_j-b_j^2/3)+
a_j^2d_j-b_j^3/27$.
Denote 
$\tilde{x_j}= a_jx_j+b_j/3$.
The problem is that from
$\tilde{x_j}^3+\epsilon \tilde{x_j} +\delta =0$
we do not know how to show that $F$ is an automorphism.

Notice that if $\epsilon =0$,
then $F$ is an automorphism
by considerations we have already seen
(Lemma \ref{lemma} and Wang's intersection theorem).
Actually, $\epsilon=0$ 
is exactly the case we have dealt with in $(2)$ above:
$b_j^2=3a_jc_j$.

Also, it may be interesting to try to prove
that the degree $\leq 3$ case implies
that the generalized Jacobian Conjecture is true;
we have not yet tried to prove this implication.

Finally, we bring the following corollary to Theorem \ref{my thm wang 2}:
\begin{corollary}\label{cor my thm wang 2}
Let 
$F:\mathbb{C}[x_1,\ldots,x_n] \to \mathbb{C}[x_1,\ldots,x_n]$ 
be a $\mathbb{C}$-algebra endomorphism 
that has an invertible Jacobian.
If there exists a subset
$A \subset \{ 1,\ldots,n \}$ of order $n-1$
and 
$\{m_j \geq 1 \}_{j \in A}$,
such that for every $j \in A$,
the minimal polynomial of $x_j^{m_j}$ over 
$\mathbb{C}(F_1,\ldots,F_n)$
has degree $\leq 2$,
then $F$ is an automorphism.
\end{corollary}

\begin{proof}
It is not difficult to adjust the proof of Theorem 
\ref{my thm wang 2} to the current case;
shortly, 
denote the minimal polynomial of $x_j^{m_j}$ over 
$\mathbb{C}(F_1,\ldots,F_n)$
by $f_j=f_j(T)$.

If $f_j$ is of degree $1$,
then 
$\mathbb{C}[x_1,\ldots,x_n] \ni x_j^{m_j} = 
-a_j \in \mathbb{C}(F_1,\ldots,F_n)$,
so by Wang's intersection theorem,
$x_j^{m_j} \in \mathbb{C}[F_1,\ldots,F_n]$,
and then by Lemma \ref{lemma},
$x_j \in \mathbb{C}[F_1,\ldots,F_n]$.

If $f_j$ is of degree $\leq 2$,
then 
$(ax_j^{m_j}+b/2)^2=b^2/4-ac \in \mathbb{C}[F_1,\ldots,F_n]$,
so by Lemma \ref{lemma},
$ax_j^{m_j}+b/2 \in \mathbb{C}[F_1,\ldots,F_n]$,
then
$x_j^{m_j} \in \mathbb{C}(F_1,\ldots,F_n)$,
so by Wang's intersection theorem
$x_j^{m_j} \in \mathbb{C}[F_1,\ldots,F_n]$,
and again by Lemma \ref{lemma},
$x_j \in \mathbb{C}[F_1,\ldots,F_n]$.
\end{proof}
\subsection{Magnus' theorem and some of its generalizations}

Inspired by Magnus' theorem, we bring the following result,
which we were able to prove thanks to 
Formanek's field of fractions theorem:

\begin{theorem}[Our Magnus' theorem]\label{my thm magnus}
Let 
$F:\mathbb{C}[x_1,x_2] \to \mathbb{C}[x_1,x_2]$ 
be a $\mathbb{C}$-algebra endomorphism 
that has an invertible Jacobian.
If $\gcd(d_1,d_2)=1$,
then $F$ is an automorphism.
\end{theorem}

\begin{proof}
{}From Formanek's field of fractions theorem,
$\mathbb{C}(F_1,F_2,x_1)
=\mathbb{C}(F_1,F_2,x_2)
=\mathbb{C}(x_1,x_2)$.
Hence, the degree of the field extension
$\mathbb{C}(F_1,F_2) \subseteq \mathbb{C}(x_1,x_2)$
equals the degree of the minimal polynomial of
$x_1$ over $\mathbb{C}(F_1,F_2)$,
$d_1$,
and 
equals the degree of the minimal polynomial of
$x_2$ over $\mathbb{C}(F_1,F_2)$,
$d_2$,
so we have $d_1=d_2$.
By our assumption that 
$\gcd(d_1,d_2)=1$,
we get that $d_1=d_2=1$,
hence
$\mathbb{C}(F_1,F_2)=\mathbb{C}(x_1,x_2)$.
Finally, Keller's theorem implies that
$F$ is an automorphism.
\end{proof} 

Inspired by the theorem of Nakai and Baba,
we show that in Theorem \ref{my thm magnus}
we can take $\gcd(d_1,d_2) \leq 2$
and still get that $F$ is an automorphism:

\begin{theorem}[Our Nakai-Baba theorem]\label{my thm nakai baba}
Let 
$F:\mathbb{C}[x_1,x_2] \to \mathbb{C}[x_1,x_2]$ 
be a $\mathbb{C}$-algebra endomorphism 
that has an invertible Jacobian.
If $\gcd(d_1,d_2) \leq 2$,
then $F$ is an automorphism.
\end{theorem}

\begin{proof}
If $\gcd(d_1,d_2)=1$, 
then Theorem \ref{my thm magnus}
says that $F$ is an automorphism.

So we assume that $\gcd(d_1,d_2)=2$.
By exactly the same arguments as in the proof of 
Theorem \ref{my thm magnus},
we have $d_1=d_2$.
Then from
$\gcd(d_1,d_2) = 2$,
we get that $d_1=d_2 = 2$.
Hence the degree of the field extension
$\mathbb{C}(F_1,F_2) \subseteq \mathbb{C}(F_1,F_2,x_1)=
\mathbb{C}(x_1,x_2)$ is $2$.
Since every separable field extension of degree $2$ is Galois,
we get that
$\mathbb{C}(F_1,F_2) \subseteq \mathbb{C}(x_1,x_2)$ 
is Galois, and we are done by the Galois case theorem
(see also Remark \ref{remark}).
\end{proof} 

In order to obtain an analogue result to Magnus' corollary,
we add an assumption on the form of the minimal polynomial
of $x_1$ over $\mathbb{C}(F_1,F_2)$.
(A separable field extension of prime degree $p > 2$ need not be Galois).
Without that additional assumption we are not able to show that 
$F$ is an automorphism,
except for the case $d_1=2$, 
which implies that 
$\mathbb{C}(F_1,F_2) \subseteq \mathbb{C}(x_1,x_2)$
is Galois.
Already when $d_1=3$, we do not know how to show that $F$ is an automorphism.
(We have discussed $d_1=3$ above, 
between Remark \ref{remark} and Corollary \ref{cor my thm wang 2}.
Notice that now $n=2$, while above $n \geq 2$).

\begin{theorem}[Our Magnus' corollary]\label{my thm magnus cor}
Let 
$F:\mathbb{C}[x_1,x_2] \to \mathbb{C}[x_1,x_2]$ 
be a $\mathbb{C}$-algebra endomorphism 
that has an invertible Jacobian.
Assume that the coefficients of the minimal
polynomial of $x_1$ over $\mathbb{C}[F_1,F_2]$
are all symmetric or skew-symmetric
with respect to the exchange involution
$\alpha$:
$\alpha(x_1) = x_2$, $\alpha(x_2) = x_1$.
If $d_1$ or $d_2$
is a prime number $p \geq 2$, 
then $F$ is an automorphism.
\end{theorem}

Denote the minimal polynomial of $x_1$ over 
$\mathbb{C}(F_1,F_2)$ by $f_1$.
By the minimal polynomial of $x_1$ over 
$\mathbb{C}[F_1,F_2]$ we mean
$f_1$ multiplied by the lcm of the denominators
of the coefficients of $f_1$
(of course, that lcm exists, since 
$\mathbb{C}[F_1,F_2]$ is a UFD).

\begin{proof}
W.l.o.g. $d_1=p$.
Denote the minimal polynomial of $x_1$ over 
$\mathbb{C}[F_1,F_2]$ by $\tilde{f_1}$.
Write
$\tilde{f_1}=c_pT^p+c_{p-1}T^{p-1}+\ldots+c_1T+c_0$
(of course,
$c_p,c_{p-1},\ldots,c_1,c_0 \in \mathbb{C}[F_1,F_2]$),
so
$c_px_1^p+c_{p-1}x_1^{p-1}+\ldots+c_1x_1+c_0 = 0$.
By our assumption, the coefficients
$c_p,c_{p-1},\ldots,c_1,c_0$
are all symmetric or skew-symmetric
with respect to the exchange involution
$\alpha$.
\begin{itemize}

\item [(1)] Symmetric case:
$0= \alpha(c_px_1^p+c_{p-1}x_1^{p-1}+\ldots+c_1x_1+c_0)
= c_p\alpha(x_1^p)+c_{p-1}\alpha(x_1^{p-1})+\ldots+c_1\alpha(x_1)+c_0
= c_px_2^p+c_{p-1}x_2^{p-1}+\ldots+c_1x_2+c_0$.
Namely, $\tilde{f_1}(x_2)=0$,
so $f_1$ is also the minimal polynomial of $x_2$
over $\mathbb{C}(F_1,F_2)$.
Hence $x_1$ and $x_2$ are conjugates 
(= both have the same minimal polynomial).
We can apply the following nice result in Galois theory:
``Let $k(a)/k$ be a separable field extension of prime degree.
Assume that there exists a conjugate $b$ to $a$ such that
$b \in k(a)$.
Then $k(a)/k$ is Galois";
several proofs for this result can be found in ~\cite{ma}.
Here: 
$\mathbb{C}(F_1,F_2)(x_1)/\mathbb{C}(F_1,F_2)$ 
is a separable field extension of prime degree $p$.
We have just obtained that $x_2$ is conjugate to $x_1$,
and 
$x_2 \in \mathbb{C}(x_1,x_2)=
\mathbb{C}(F_1,F_2)(x_1)$
(by Formanek's field of fractions theorem).
Therefore, 
$\mathbb{C}(x_1,x_2)/\mathbb{C}(F_1,F_2)
=\mathbb{C}(F_1,F_2)(x_1)/\mathbb{C}(F_1,F_2)$ 
is Galois,
and we are done.

\item [(2)] Skew-symmetric case:
$0= \alpha(c_px_1^p+c_{p-1}x_1^{p-1}+\ldots+c_1x_1+c_0)
= -c_p\alpha(x_1^p)-c_{p-1}\alpha(x_1^{p-1})-\ldots-c_1\alpha(x_1)-c_0
= -c_px_2^p-c_{p-1}x_2^{p-1}-\ldots-c_1x_2-c_0$.
Namely, $\tilde{f_1}(x_2)=0$,
so $f_1$ is also the minimal polynomial of $x_2$
over $\mathbb{C}(F_1,F_2)$.
Hence $x_1$ and $x_2$ are conjugates.
Again, by the nice result in Galois theory ~\cite{ma},
we get that 
$\mathbb{C}(x_1,x_2)/\mathbb{C}(F_1,F_2)$
is Galois.
\end{itemize}

\end{proof}

Then, almost analogously to the theorem of Appelgate-Onishi-Nagata 
~\cite[Theorem 10.2.26]{essen book} we have the following theorem
('Almost analogously', because we do not know what happens when
$\gcd(d_1,d_2) \in \{4,6,8 \}$).

\begin{theorem}[Our Appelgate-Onishi-Nagata theorem]\label{my thm magnus generalized}
Let 
$F:\mathbb{C}[x_1,x_2] \to \mathbb{C}[x_1,x_2]$ 
be a $\mathbb{C}$-algebra endomorphism 
that has an invertible Jacobian.
Assume that the coefficients of the minimal
polynomial of $x_1$ over $\mathbb{C}[F_1,F_2]$
are all symmetric or skew-symmetric
with respect to the exchange involution
$\alpha$:
$\alpha(x_1) = x_2$, $\alpha(x_2) = x_1$.
If $\gcd(d_1,d_2) =1$ or 
$\gcd(d_1,d_2) = p$,
where $p \geq 2$ is a prime number, 
then $F$ is an automorphism.
\end{theorem}

Notice that Theorem \ref{my thm magnus generalized}
includes the previous theorems: \ref{my thm magnus},
Theorem \ref{my thm nakai baba}
and
Theorem \ref{my thm magnus cor}.

\begin{proof}
If $\gcd(d_1,d_2) =1$, 
then from Theorem \ref{my thm magnus}
$F$ is an automorphism
(even without the additional assumption on the coefficients of
the minimal polynomial of $x_1$ 
over $\mathbb{C}[F_1,F_2]$).

So assume that $\gcd(d_1,d_2) = p$,
where $p \geq 2$ is a prime number.
{}From Formanek's field of fractions theorem,
$\mathbb{C}(F_1,F_2,x_1)=
\mathbb{C}(F_1,F_2,x_2)=
\mathbb{C}(x_1,x_2)$,
so $d_1=d_2$.
Then,
$d_1=d_2=\gcd(d_1,d_2) = p$.
Therefore we can apply 
Theorem \ref{my thm magnus cor}
and get that $F$ is an automorphism.
\end{proof}

\subsubsection{Magnus' theorem in higher dimensions}

Magnus' theorem and its corollary and its generalizations
are for $\mathbb{C}[x_1,x_2]$.
We are not familiar with similar results for
$\mathbb{C}[x_1,\ldots,x_n]$, $n \geq 3$,
probably because it is difficult to prove such results.

Luckily, it is not difficult to adjust our
Theorem \ref{my thm magnus} to higher dimensions:
\begin{theorem}\label{my thm magnus n}
Let 
$F:\mathbb{C}[x_1,\ldots,x_n] \to \mathbb{C}[x_1,\ldots,x_n]$ 
be a $\mathbb{C}$-algebra endomorphism 
that has an invertible Jacobian.
If $\gcd(d_u,d_v)=1$, $1 \leq u \neq v \leq n$,
for each of the $n \choose 2$ pairs of degrees,
then $F$ is an automorphism.
\end{theorem}

When $n=2$, this is just Theorem \ref{my thm magnus}.

\begin{proof}
Denote by $D$ the degree of the field extension
$\mathbb{C}(F_1,\ldots,F_n) \subseteq
\mathbb{C}(x_1,\ldots,x_n)$.
Formanek's field of fractions theorem says that
$\mathbb{C}(F_1,\ldots,F_n,x_{m_1},\ldots,x_{m_{n-1}})=
\mathbb{C}(x_1,\ldots,x_n)$,
for any 
$m_1,\ldots,m_{n-1} \in \{1,\ldots,n \}$.
Then, since $\gcd(d_u,d_v)=1$, $1 \leq u \neq v \leq n$,
for each of the $n \choose 2$ pairs of degrees,
it follows that $D= d_{m_1} \cdots d_{m_{n-1}}$,
for any 
$m_1,\ldots,m_{n-1} \in \{1,\ldots,n \}$.
Then clearly, $d_u=d_v$, $1 \leq u \neq v \leq n$,
for each of the $n \choose 2$ pairs of degrees
(so $d_1=\ldots=d_n$).
{}From
$\gcd(d_u,d_v)=1$, we immediately get
$d_u=d_v=1$, 
so $d_1=\ldots=d_n=1$.
It follows that $D=1$,
hence
$\mathbb{C}(F_1,\ldots,F_n) = \mathbb{C}(x_1,\ldots,x_n)$,
and we are done by Keller's theorem.
\end{proof}

In view of Theorem \ref{my thm magnus n}
and Magnus' theorem we conjecture:
\begin{conjecture}[$n$-GCD Conjecture]
Let 
$F:\mathbb{C}[x_1,\ldots,x_n] \to \mathbb{C}[x_1,\ldots,x_n]$ 
be a $\mathbb{C}$-algebra endomorphism 
that has an invertible Jacobian.
Denote the degree of $F_i$
by $l_i$,
$1 \leq i \leq n$.
If $\gcd(l_u,l_v)=1$, $1 \leq u \neq v \leq n$,
for each of the $n \choose 2$ pairs of degrees,
then $F$ is an automorphism.
\end{conjecture}

Obviously, the theorem of Magnus says that the $2$-GCD Conjecture is true.
However, the $3$-GCD Conjecture seems more difficult to prove.
The proof of the $2$-GCD Conjecture uses some results about Newton polygons
in dimension $2$, see ~\cite[Section 10.2]{essen book}, 
so a proof of the $3$-GCD Conjecture (if true)
should probably use analogous results about Newton polygons
in dimension $3$.
We guess that if the $3$-GCD Conjecture is true,
then the $n$-GCD Conjecture is also true for all $n$.

\bibliographystyle{plain}

\end{document}